\documentclass[a4paper,10pt]{article}
\usepackage[utf8]{inputenc}
\usepackage{amssymb,amsmath,amsthm}
 \usepackage{array,multirow}
\usepackage[table]{xcolor}
\newcommand\gry{\cellcolor{black!10!white}}
\definecolor{grey}{RGB}{132, 132, 132}
\usepackage{hyperref}
\usepackage{currfile}
\usepackage{navigator}
\IfFileExists{\currfilename}{\embeddedfile{sourcefile}{\currfilename}}{}
\hypersetup{pdftitle={Do K33-Free Latin Squares Exist?},
             pdfsubject={Do K33-Free Latin Squares Exist?},
             pdfauthor={A. Krotov, D. Krotov}}

\date{}

\definecolor{brown}{RGB}{255, 0, 255}
\title{Do $K_{3,3}$-Free Latin Squares Exist?%
\thanks{This is a post-peer-review, pre-copyedit version of the article published in Discrete Mathematics 350 (2027), 115402, \url{https://doi.org/10.1016/j.disc.2026.115402}}%
}
\author{Aleksandr D. Krotov%
\thanks{Yandex, Moscow 119021, Russia.
E-mail: san.crot@yandex.ru}%
,
Denis S. Krotov%
\thanks{School of Mathematical Sciences, Hebei Key Laboratory of Computational Mathematics and Applications, Hebei
Normal University, Shijiazhuang 050024, P. R. China}
\thanks{Sobolev Institute of Mathematics, Novosibirsk 630090, Russia. E-mail: dk@ieee.org}%
}

\newtheorem{lemma}{Lemma}
\newtheorem{proposition}{Proposition}
\newtheorem{theorem}{Theorem}
\newtheorem{corollary}{Corollary}
\theoremstyle{remark}
\newtheorem{remark}{\bf Remark}
\newtheorem{problem}{\bf Problem}

\begin{document}

\maketitle

\begin{abstract}
We discuss the problem of the existence of latin squares without a substructure consisting of six elements $(r_1,c_2,l_3)$,  $(r_2,c_3,l_1)$,  $(r_3,c_1,l_2)$, $(r_2,c_1,l_3)$,  $(r_3,c_2,l_1)$,  $(r_1,c_3,l_2)$. Equivalently, the corresponding latin square graph does not have an induced subgraph isomorphic to $K_{3,3}$. The exhaustive search [Brouwer, Wanless. Universally noncommutative loops. 2011] shows that no such latin squares exist for orders $3$, $4$, $5$, $6$, $7$, $9$, $10$, $11$ and there are only two $K_{3,3}$-free latin squares of order $8$, up to equivalence. We repeat the search, establishing also the number of $K_{3,3}$-free latin $m$-by-$n$ rectangles for each $m$ and $n$ less than or equal to $11$. As a switched combination of two orthogonal latin squares of order $8$, we construct a $K_{3,3}$-free (universally noncommutative) latin square of order $16$.

We also consider a similar problem for orthogonal latin squares, proving that there are both $K_{4,4}$-free and non-$K_{4,4}$-free linear pairs of orthogonal latin squares for each odd prime-power order larger than~$5$.

\end{abstract}

\section{Introduction}

An $m \times n$ \emph{latin rectangle} (if $m=n$, a \emph{latin square} of order $n$) 
$S$
is a set of $mn$ ordered triples from 
${R} \times {C} \times {L}$, 
where ${R}$, ${C}$, ${L}$ are some sets such that 
$m=|{R}|\le |{C}| = |{L}|=n$, 
such that no two triples from $S$ differ in exactly one position.
Without loss of generality, one can take ${R}=\{0,\ldots,m-1\}$ and ${C}={L}=\{0,\ldots,n-1\}$;
however, sometimes it is convenient to construct latin rectangles or squares 
with specially chosen sets $R$, $C$, $L$.
It is typical to imagine a latin rectangle (square) $S$ as an $m\times n$ 
table whose rows are indexed by ${R}$, columns are indexed by ${C}$, 
and each cell $(r,c)$, $r\in{R}$, $c\in{C}$, 
contains an element $l$ from ${L}$ (we will write $S(r,c)=l$ in this case)
if $(r,c,l) \in S$. Two latin squares $A$ and $B$
are called 
\emph{orthogonal} if
$(r,c,l),(r',c',l)\in A$ and $(r,c,g),(r',c',g)\in B$
imply $(r,c) = (r',c')$.

Given a latin rectangle (square) $S$, the graph $\Gamma(S)$ (if $m=n$, a \emph{latin square graph})
has $S$ as the vertex set, where two distinct vertices $(r,c,l)$, $(r',c',l')$ are adjacent if and only if $r=r'$, $c=c'$, or $l=l'$.
A set of $n$ vertices having the same value in the first position forms a clique of $\Gamma(S)$, 
which is called a \emph{row} of the latin rectangle (square)~$S$.
Similarly, fixing the value in the second (third) position,
we obtain a clique of size $m$, called a  \emph{column} (respectively,  \emph{letter}) of~$S$.

Two latin rectangles (squares) $A\subset {R} \times {C} \times {L}$ and $A'\subset {R'} \times {C'} \times {L'}$ are \emph{isotopic} if there are bijections $\rho:R\to R'$, $\gamma:C\to C'$, and $\lambda:L\to L'$
such that $A' = \{(\rho r,\gamma c, \lambda l)\,:\,(r,c,l) \in A\}$.
Two latin rectangles (squares) $A$, $A'$ are \emph{paratopic} if 
$A'$ is isotopic to one of the six \emph{conjugates} 
(also called \emph{parastrophes})
of $A$: 
\begin{eqnarray*}
A_{0,1,2} &=& A, \\
A_{0,2,1} &=& \{(r,l,c)\,:\,(r,c,l) \in A\}, \\
A_{1,0,2} &=& \{(c,r,l)\,:\,(r,c,l) \in A\}, \\
A_{2,0,1} &=& \{(c,l,r)\,:\,(r,c,l) \in A\}, \\
A_{2,1,0} &=& \{(l,c,r)\,:\,(r,c,l) \in A\}, \\
A_{1,2,0} &=& \{(l,r,c)\,:\,(r,c,l) \in A\}
\end{eqnarray*}
(of course, in the nonsquare case, only $A_{0,1,2}$ or $A_{0,2,1}$ can be isotopic to $A$; the other four are not latin rectangles, formally).

The isotopy and paratopy are well-defined equivalences;
the corresponding equivalence classes of latin rectangles (squares) are called
the \emph{isotopy classes} and the \emph{main classes}.
The group of all isotopisms (i.e., triples $(\rho,\gamma,\lambda)$, with composition as the group operation) that
send a latin square~$A$ to itself, 
$$A = \{(\rho r,\gamma c, \lambda l)\,:\,(r,c,l) \in A\},$$
is called the \emph{autotopism group}
of~$A$.
The group of all paratopisms (i.e., quadruples 
$(\rho,\gamma,\lambda,\pi)$ where $\pi$ is a permutation of $\{0,1,2\}$) that 
send a latin square~$A$ to itself, 
$$A = \{(\rho r,\gamma c, \lambda l)\,:\,(r,c,l) \in A_{\pi(0),\pi(1),\pi(2)}\},$$
is called the \emph{paratopism group} of~$A$.
Note that two latin squares are paratopic if and only if their graphs are isomorphic, 
while this is not always true for $m\times n$ latin rectangles
if $2m\le n$.

We say that a latin  rectangle (square) $S$ is \emph{$K_{3,3}$-free} 
if $\Gamma(S)$ has no induced subgraph isomorphic to $K_{3,3}$,
the complete bipartite graph with two parts of size~$3$. (Recall that an \emph{induced} subgraph of a graph is another graph, formed from a subset of the vertices of the graph and all of the edges, from the original graph, connecting pairs of vertices in that subset.)
It is easy to see the following proposition (a proof can be found in the appendix), which illustrates the fact that
the {$K_{3,3}$-free} latin rectangles are the latin rectangles that avoid the six-element pattern
 $$ \begin{array}{ccc}
 \cdot & \mathbf C & \mathbf B \\
  \mathbf C  & \cdot  & \mathbf A \\
  \mathbf B  & \mathbf A & \cdot
 \end{array}\qquad $$
 in any three rows and three columns (the order of those rows and columns can be different from the order of the corresponding rows and columns in the pattern).
\begin{proposition}\label{p:six}
Six distinct elements $x$, $x'$, $x''$, $y$, $y'$, $y''$ of a latin rectangle $S$ induce a complete bipartite
subgraph of \,$\Gamma(S)$ with parts $\{x, x', x''\}$ and $\{y, y', y''\}$ if and only if
\begin{align*}
 \{x,x',x''\}&=\{(r_1,c_2,l_3),(r_2,c_3,l_1),(r_3,c_1,l_2)\}
 \mbox{ and }\\ 
 \{y, y', y''\}&=\{(r_1,c_3,l_2), (r_2,c_1,l_3), (r_3,c_2,l_1)\}
 \end{align*}
for some $r_1$, $r_2$, $r_3$, $c_1$, $c_2$, $c_3$, $l_1$, $l_2$, $l_3$.
\end{proposition}

The current paper is devoted to the following question,
discussed at the open-problem section of
the International Conference and PhD-Master Summer School on Graphs and Groups, Representations and Relations (August 2018, Novosibirsk).

\begin{problem}\label{prob:1}
 For which $n$ are there $K_{3,3}$-free latin squares of order~$n$?
\end{problem}

The following elementary proposition illustrates the problem, stating that the value tables of finite groups
are not $K_{3,3}$-free (trivially, such a table is a latin square).
A proof can be found in the appendix.

\begin{proposition}\label{p:commut}
 Given a finite group $(R,\cdot)$ of order at least~$3$, the value table $S=\{(r,c,r\cdot c)\mid r,c \in R\}$ is not $K_{3,3}$-free.
\end{proposition}

As for many other classes of combinatorial objects,
the study of latin squares with or without certain patterns
is a popular direction of research.
The most popular patterns are latin subsquares
(see, e.g., \cite{Heinrich:1980}, \cite{Heinrich:Ch4}, \cite{MWW:2007}, \cite{KSSS:2023}, \cite{AllsopWanless:subsquares}),
trades (see the survey \cite{Cav:rev}),
transversals
(see the surveys \cite{Wanless:survey2011}, \cite{Montgomery:2024})
and their multifold generalizations, plexes
(see, e.g., \cite{EganWanless:2008}, \cite{Pula:2011}, \cite{Taranenko:2018:plex}).

Problem~\ref{prob:1} has already been considered in~\cite{BroWan:2011}
(in a different context 
of universally noncommutative quasigroups,
which is essentially the same 
as $K_{3,3}$-free latin squares), 
where it was established computationally 
that for $3\le n\le 11$, there are only two main classes of such squares, both with $n=8$.
In this paper, we repeat the calculation, paying more attention
to intermediate results, in particular, to $K_{3,3}$-free latin rectangles,
and make a step forward by constructing a $K_{3,3}$-free latin square
of order~$16$.
Additionally, we consider a similar problem for orthogonal latin squares,
establishing the existence of an infinite number of $K_{4,4}$-free linear pairs of orthogonal latin squares, which contrasts with the case of linear latin squares
(Proposition~\ref{p:commut}).

In Section~\ref{s:motiv}, we show the place of Problem~1 in the general context of trades
of combinatorial configurations, or, more specifically, of transversal trades in latin squares, and in the context of studying the minimum supports of eigenfunctions of graphs.
In Section~\ref{s:9}, we report the results of an exhaustive search that
repeats the results of~\cite{BroWan:2011} saying that among
the latin squares of order from~$3$ to~$11$, 
there are only two nonparatopic $K_{3,3}$-free latin squares, both of order~$8$.
Moreover, we count the $K_{3,3}$-free latin $m\times n$ rectangles for all~$m$ and~$n$
not more than~$11$ and, in particular, find that 
$K_{3,3}$-free latin $(n-1)\times n$ rectangles are also very rare objects.
Apart from the order~$8$, we have found one $4\times 5$ rectangle,
one $8\times 9$ rectangle, both with nice cyclic structures,
and (by an additional nonexhaustive search)
a less symmetric $11\times 12$ latin rectangle avoiding $K_{3,3}$.
In Section~\ref{s:16}, trying to generalize the $8\times 8$ examples,
we construct a $K_{3,3}$-free latin $16\times 16$ square
as a switched combination of two orthogonal $8\times 8$ latin squares.
For every pair of orthogonal $8\times 8$ latin squares, an admissible switching
to combine them into a $K_{3,3}$-free $16\times 16$ square is searched
as a solution of a system of linear equations over GF$(2)$, and it is notable
that only one pair of such squares (among almost twelve thousand) admits a solution.
For every order larger than~$11$ and different from~$16$,
the problem of the existence of a $K_{3,3}$-free latin square remains open.
In Section~\ref{s:ort},
we discuss a similar problem for orthogonal latin squares and prove (Corollary~\ref{cor:K44}) that
in the class of linear pairs of orthogonal latin squares, there are an infinite number
of both $K_{4,4}$-free and non-$K_{4,4}$-free pairs.

\section{Motivation}\label{s:motiv}
\subsection{Clique designs and trades; transversals}\label{s:trades}

Assume we have a regular graph $\Gamma$ with a nonempty collection~$A$ of cliques such that every edge lies in a constant number of cliques from~$A$.
We say that a set $T$ of vertices is a \emph{clique design} (more generally, a \emph{$\lambda$-fold clique design}) if every clique from $A$ contains exactly $1$ element (respectively, exactly $\lambda$ elements) of~$T$.
A pair $(T_+,T_-)$ of nonempty disjoint vertex sets of $\Gamma$ is called a \emph{clique trade} (alternatively, bitrade) if for every clique $a$ from $A$ one has 
$|a\cap T_+|=|a\cap T_-|$.
Sometimes, an additional condition $|a\cap T_+|\le 1$
is required.
The values $|T_+|$ and $|T_-|$ coincide and are called the \emph{volume}
of the trade. 
If $T'$ and $T''$ are ($\lambda$-fold) clique designs, then the pair
$(T_+,T_-)$, where $T_+=T'\backslash T''$, $T_-=T''\backslash T'$, 
is a clique trade.
In this case, one can say that $T''$ is obtained from $T'$ by switching the trade 
$(T_+,T_-)$.
However, not every clique trade can be represented as the difference pair of two clique designs, and studying trades is of independent interest even for graphs where clique designs do not exist.

The concept of clique designs covers many classes of combinatorial configurations, such as $t$-$(v,t+1,\lambda)$ designs (in Johnson graphs)
and their subspace analogs (in Grassmann graphs), distance-$2$ MDS codes (equivalently, latin hypercubes or multiary quasigroups; in Hamming graphs),
binary $1$-perfect codes (in halved hypercube graphs), distance-$2$ MRD codes
(in bilinear forms graphs), Cameron--Liebler line classes (in the complement of a diameter-$2$ Grassmann graph). 
Studying the corresponding trades is very important in the study of each concrete class of combinatorial configurations.
In many cases, good lower bounds on the number of different configurations 
of the same parameters are based on switching independent trades (see e.g. examples in the survey \cite{Krotov:trades}).
Such an approach is more effective when one switches trades of small volume
rather than of large volume.
From this point of view, an important question is how small a trade can be in a concrete graph (with a given system of cliques) or in some class of graphs.

If $\Gamma$ is a latin square graph and the clique system $A$ 
consists of the rows, columns, and letters,
then the clique designs are known as the 
\emph{transversals} of the corresponding latin square
(the $\lambda$-fold clique designs are called \emph{$\lambda$-plexes}).
It is natural in this case to refer to the clique trades as the \emph{transversal trades}.
It is not difficult to find that 
the volume of a transversal trade cannot be $1$ or $2$, but can be equal to $3$
if and only if the latin square graph
has an induced $K_{3,3}$-subgraph.
Based on switching transversal trades of volume $3$, a
constructive asymptotic lower bound on the number of transversals 
in a latin square was established in~\cite{Pot:2016:trans}
(a better bound, which is  asymptotically tight,
was further established in~\cite{EMM:2019},
using another approach).

In the context described above, Problem~\ref{prob:1} can be viewed as a particular question concerning the study
of the spectrum of volumes of transversal trades (or, more generally, clique trades).

\subsection{Minimum supports of eigenfunctions}\label{s:eigen}

An \emph{eigenfunction} corresponding to an \emph{eigenvalue} $\theta$ 
of a graph $\Gamma=(V,E)$ is a real-valued function
$f:V\to\mathbb{R}$ such that $f\not\equiv 0$ and
$\sum_{y:\{x,y\}\in E} f(y)=\theta  f(x)$ 
for all~$x$ from~$V$.
Many classes of combinatorial configurations and their trades are representable
as functions with a pure eigenspectrum; often (including all examples mentioned in Section~\ref{s:trades}),
as eigenfunctions satisfying additional discrete constraints on their values.
In this context, the study of the smallest volume of trades motivates the study
of the smallest support of eigenfunctions, for given graph and eigenvalue.
For different graphs, this question was considered in
\cite{KMP:16:trades},
\cite{Sotnikova:2018}, \cite{Val:min2}, \cite{VorVal:Hamm}, \cite{VMV:Johnson}, \cite{GKSV:Paley},
\cite{Sotnikova:bil}, \cite{SotVal:survey2021}, \cite{Adriaensen:wd}.

The general results of \cite{KMP:16:trades} for distance-regular graphs, applied to a latin square graph, imply the following.
\begin{proposition}\label{p:6}
 An eigenfunction corresponding to the {eigenvalue} $-3$ of a latin square graph has at least $6$ nonzeros. Moreover, 
 if the number of nonzeros  is exactly $6$, 
 then they induce a subgraph 
 isomorphic to $K_{3,3}$.
\end{proposition}

So, Problem 1 is a special case of the minimum-support problem for eigenfunctions of graphs.
Moreover, in the case of strongly regular graphs,
the eigenfunctions with the number of nonzeros attaining the so-called
weight-distribution lower bound correspond to induced $K_{s,s}$-subgraphs,
where $s$ is derived from the parameters of the strongly regular graph and the eigenvalue. In a sequence of papers
\cite{GKSV:Paley,GShY:2023,GorPan:2024,GorYip:2024,DeBruyn2025,EvGoSh:2025,Adriaensen:wd}, it is proved that the weight-distribution bound is attained for different families of strongly regular graphs (in the recent work \cite{Adriaensen:wd}, also the non-attainability of this bound for some graphs is proved). $K_{3,3}$-free latin squares and $K_{4,4}$-free pairs of orthogonal latin squares
(constructed in Section~\ref{s:ort}) give examples of strongly regular graphs
whose eigenfunctions do not attain the weight-distribution lower bound on the cardinality of the support.

\begin{remark}
 If one treats an eigenspace of a graph as a linear code,
then the minimum size of the support of an eigenfunction from this space 
is the minimum Hamming distance of this code.
\end{remark}

\section{Small orders}\label{s:9}

In this section, we report the results of a computer search for $K_{3,3}$-free latin rectangles
of order at most~$11$.

We use the standard algorithm of classification of latin rectangles, see e.g. \cite{KO:alg}, 
adapted to the $K_{3,3}$-free case.
We first describe the general approach, without the $K_{3,3}$-free condition.
The algorithm is recursive in~$m$;
for each given $m\times n$ latin rectangle $S_{m,n}\subset R_m\times C \times L$,  
it produces all possible continuations to a latin rectangle of size $(m+1)\times n$.
\begin{itemize}
\item[(i)] Let $m$ be the new row index, $m\not\in R_m$. 
 Among all $n^2$ elements $(m,c,l)$, $c\in C$, $l\in L$,
 we choose $(n-m)n$ elements $(m,c,l)$, \emph{candidates}, such that 
 $S_{m,n}$ does not contain any triple $(r,c,l)$, $r\in R_m$.
\item[(ii)] On the set of candidates, we define the compatibility graph $\Delta$ as follows.
 Two vertices $(m,c',l')$, $(m,c'',l'')$ are compatible (i.e., adjacent in $\Delta$)
 if and only if $c'\ne c''$ and $l'\ne l''$.
\item[(iii)] Using specialized software
(e.g., \texttt{cliquer} \cite{cliquer})
 we find all cliques~$Q$ of size~$n$ in the graph~$\Delta$.
 The unions $S_{m,n}\cup Q$ are all possible $(m+1)\times n$
 latin rectangles that continue $S_{m,n}$.
\item[(iv)] Using the software \cite{nauty2014} for dealing with graph isomorphism,
we find all nonparatopic representatives of the found latin rectangles, to use them at the next iterative step or to represent the result of the classification.
\end{itemize}
It happens that this approach can be easily adapted to classify the $K_{3,3}$-free latin rectangles.
Indeed, if six elements of a latin rectangle $S_{m+1,n}$ form a forbidden pattern 
(i.e., induce $K_{3,3}$ in $\Gamma(S_{m+1,n})$), then either they all lie in the first $m$ rows, 
or exactly two of them belong to the last row.
The first possibility is excluded if the latin rectangle $S_{m,n}$, constructed at the previous iteration, is $K_{3,3}$-free.
To avoid the second possibility, we forbid the corresponding pairs in the last row by excluding them from the edge set of~$\Delta$.
So, we replace~(ii) by
\begin{itemize}
\item[(ii')] On the set of candidates, we define the compatibility graph~$\Delta$ as follows.
 Two vertices $(m,c',l')$, $(m,c'',l'')$ are compatible (i.e., adjacent in~$\Delta$)
 if and only if $c'\ne c''$, $l'\ne l''$,
 \emph{and, moreover,  
 there are no $c$, $r'$, $r''$
 such that $(r',c,l'),(r'',c,l'')\in S_{m,n}$
 and $S_{m,n}(r',c'')=S_{m,n}(r'',c')$.}
\end{itemize}
Applying (i)-(ii')-(iii)-(iv) to representatives
$S_{m,n}$ of the main classes of
$K_{3,3}$-free  $m\times n$ latin rectangles,
we obtain representatives of all main classes
of
$K_{3,3}$-free $(m+1)\times n$ latin rectangles.
If the resulting set is nonempty, the standard approach \cite[10.2]{KO:alg}
based on the orbit-stabilizer 
theorem is applied to validate the results.
The approach is based on the fact that the total number of latin rectangles
from the considered class (see Table~\ref{t:3-10}) can be found in two ways:
firstly, using the total number of solutions found at step~(iii)
and the orders of the paratopism groups of the representatives~$S_{m,n}$;
and secondly, using only nonparatopic solutions (step~(iv))
and the orders of their paratopism groups.
Many kinds of random or systematic errors in the calculation result in
the disagreement in the values found with these two ways; 
conversely, agreement between these values indicates
that the probability of an error is convincingly low.

Table~\ref{t:3-9} shows the number of nonparatopic
$K_{3,3}$-free latin rectangles for small values of~$n$.
\begin{table}
 $$
\begin{array}{|c|*{10}{@{\ }c@{\ }|}}\hline
 m
   & n{:}\ 3 & 4 & 5 & 6 & 7 & 8 & 9 & 10 & 11 & 12 \\ \hline
 3 & 0 & 0 & 1 & 4 & 12 & 67  & 418   & 3871     & 41085 & 500842 \\[-0.9ex]
   & \scriptstyle(0) & \scriptstyle(0) & \scriptstyle(1) & \scriptstyle(4) & \scriptstyle(15) & \scriptstyle(92)  & \scriptstyle(692)   &  \scriptstyle(7067)  &   \scriptstyle(79254)  &  \scriptstyle(987170) \\
 4 &   & 0 & 1 & 5 & 15 & 412 & 21493 & 2365871  & 319053308     & 51526117503 \\[-0.9ex]
   &   & \scriptstyle(0) & \scriptstyle(1) & \scriptstyle(6) & \scriptstyle(24) & \scriptstyle(701) & \scriptstyle(42127) &  \scriptstyle(4719624) &  \scriptstyle(637928108)    & \scriptstyle(103049459851) \\
 5 &   &   & 0 & 0 & 3  & 100 & 13123 & 12140467 & 17661934154 & \\[-0.9ex]
   &   &   & \scriptstyle(0) & \scriptstyle(0) & \scriptstyle(5)  & \scriptstyle(139) & \scriptstyle(25492) &  \scriptstyle(24247157) & \scriptstyle(\text{35322291492})      & \\
 6 &   &   &   & 0 & 0  & 19  & 134   & 388165   & 2846897941 &  \\[-0.9ex]
   &   &   &   & \scriptstyle(0) & \scriptstyle(0)  & \scriptstyle(25)  & \scriptstyle(254) & \scriptstyle(772702) & \scriptstyle(\text{5693554556}) &  \\
 7 &   &   &   &   & 0  & 3   & 5     & 17121    & 1766681 & \\[-0.9ex]
   &   &   &   &   & \scriptstyle(0)  & \scriptstyle(4)   & \scriptstyle(8)     &  \scriptstyle(33571) & \scriptstyle(\text{3526719}) & \\
 8 &   &   &   &   &    & 2   & 1     & 2036     &  192     & \\[-0.9ex]
   &   &   &   &   &    & \scriptstyle(4) & \scriptstyle(1)  &  \scriptstyle(3867)  & \scriptstyle(\text{352}) & \\
 9 &   &   &   &   &    &     & 0     & 0        & 0      & \\\hline
\end{array}
$$
\caption{The number of main classes (isotopy classes) of $K_{3,3}$-free $m\times n$ latin rectangles, $n\le 11$}
\label{t:3-9}
\end{table}
 \begin{table}
 $$
\begin{array}{|c||c|c|c||c||c|c|c|c|c|}\hline
 m & n=5                 & n=6                   & n=7                       \\ \hline
 3 & \scriptstyle 14400  & \scriptstyle 3110400  & \scriptstyle 1270080000
 \\[-0.3em] 
 4 & \scriptstyle 28800  & \scriptstyle 12441600 & \scriptstyle 10059033600
 \\[-0.3em] 
 5 & \scriptstyle 0      & \scriptstyle 0        & \scriptstyle 6706022400       \\[-0.3em]
\end{array}
$$
 $$
\begin{array}{|c||c|c|c|c|c|c|c|c|c|}\hline
 m & n=8                         & n=9                             & n=10      \\ \hline
 3 & \scriptstyle 622644019200   & \scriptstyle 496045696204800    & \scriptstyle 535089377894400000     \\[-0.3em]
 4 & \scriptstyle 20105061580800 & \scriptstyle 129602900651212800 & \scriptstyle 1484846226391449600000  \\[-0.3em]
 5 & \scriptstyle 17273088000000 & \scriptstyle 395432827969536000 & \scriptstyle 38270332174899363840000  \\[-0.3em]
 6 & \scriptstyle 3267661824000  & \scriptstyle 16981697101824000  & \scriptstyle 7179764880140697600000    \\[-0.3em]
 7 & \scriptstyle 2048385024000  & \scriptstyle 1769804660736000   & \scriptstyle 2107042044967157760000     \\[-0.3em]
 8 & \scriptstyle 2048385024000  & \scriptstyle 663676747776000    & \scriptstyle 1746637917726965760000      \\[-0.3em]
\end{array}
$$
 $$
\begin{array}{|c||c|c|c||c||c|c|c|c|c|}\hline
 m & n=11                & n=12                                \\ \hline
 3 & \scriptstyle 750438010707517440000  & \scriptstyle 1353254582337359708160000
 \\[-0.3em] 
 4 & \scriptstyle  24386762038751602237440000  &\scriptstyle  567430482164774579011584000000
 \\[-0.3em] 
 5 & \scriptstyle 6753604569381729842872320000      &  \\[-0.3em]
 6 & \scriptstyle 6531249036579810688696320000 &\\[-0.3em]
 7 & \scriptstyle 28299772584317096755200000 &\\[-0.3em]
 8 & \scriptstyle 21452112008596684800000 &
\end{array}
$$
\caption{The total number of different $K_{3,3}$-free $m\times n$ latin rectangles}
\label{t:3-10}
\end{table}

As we can see, the only small ($\le 11$) order for which 
$K_{3,3}$-free
squares exist is~$8$, excluding the trivial cases of order~$1$ and~$2$.
The two found nonparatopic squares are shown in Fig.~\ref{t:8}.
\begin{figure}[ht]
$$
\newcommand\0{\boldsymbol 0}
\newcommand\1{\boldsymbol 1}
\newcommand\2{\boldsymbol 2}
\newcommand\3{\boldsymbol 3}
\newcommand\4{\boldsymbol 4}
\newcommand\5{\boldsymbol 5}
\newcommand\6{\boldsymbol 6}
\newcommand\7{\boldsymbol 7}
\begin{array}{|@{\ }c@{\ }|@{\ }c@{\ }|@{\ }c@{\ }|@{\ }c@{\ }|}
 \hline
  0 & 2 & 4 & 6 \\\hline
  2 & 0 & 6 & 4 \\\hline
  4 & 6 & 0 & 2 \\\hline
  6 & 4 & 2 & 0 \\ \hline
 \end{array}\,,
\ \ 
\begin{array}{|@{\ }c@{\ }|@{\ }c@{\ }|@{\ }c@{\ }|@{\ }c@{\ }|}
 \hline
  1 & 7 & 3 & 5 \\\hline
  3 & 5 & 1 & 7 \\\hline
  5 & 3 & 7 & 1 \\\hline
  7 & 1 & 5 & 3 \\ \hline
 \end{array}
 \ \Longrightarrow
 \ 
  \begin{array}{|@{\ }c@{\ }|@{\ }c@{\ }|@{\ }c@{\ }|@{\ }c@{\ }|}
\hline
         \gry  1 \  0 & \gry  7 \  2 &  4 \ 3 &  6 \ 5 \\
         \gry  0 \  1 & \gry  2 \  7 &  3 \ 4 &  5 \ 6 \\  
\hline
          2 \ 3 &  0 \ 5 & \gry  1 \  6 & \gry  7 \  4 \\ 
          3 \ 2 &  5 \ 0 & \gry  6 \  1 & \gry  4 \  7 \\ 
\hline
          4 \ 5 & \gry  3 \  6 &  0 \ 7 & \gry  1 \  2 \\  
          5 \ 4 & \gry  6 \  3 &  7 \ 0 & \gry  2 \  1 \\
\hline                        
         \gry  7 \  6 &  4 \ 1 & \gry  5 \  2 &  0 \ 3 \\ 
         \gry  6 \  7 &  1 \ 4 & \gry  2 \  5 &  3 \ 0 \\ \hline
   \end{array}\,,
   \ 
 \ 
  \begin{array}{|@{\ }c@{\ }|@{\ }c@{\ }|@{\ }c@{\ }|@{\ }c@{\ }|}
   \hline
   \gry 1 \  0 &   2 \  7 &  4 \  3 &  6 \  5 \\ 
   \gry 0 \  1 &   7 \  2 &  3 \  4 &  5 \  6 \\ 
\hline
    2 \  3 &   0 \  5 &  6 \  1 & \gry 7 \  4 \\
    3 \  2 &   5 \  0 &  1 \  6 & \gry 4 \  7 \\  
\hline
    4 \  5 &  \gry 3 \  6 &  0 \  7 &  2 \  1 \\
    5 \  4 &  \gry 6 \  3 &  7 \  0 &  1 \  2 \\  
\hline
    6 \  7 &   4 \  1 & \gry 5 \  2 &  0 \  3 \\ 
    7 \  6 &   1 \  4 & \gry 2 \  5 &  3 \  0 \\ \hline
   \end{array}
$$
\caption{The two $K_{3,3}$-free $8\times 8$ latin squares and related orthogonal $4\times 4$ latin squares}
\label{t:8}
\end{figure}
Both squares can be obtained by combining two
orthogonal $4\times 4$ latin squares, see details in Section~\ref{s:16}.
As noted in~\cite{BroWan:2011}, each of the two squares is a G-loop isotope; equivalently \cite[Theorem~1]{KP:2016:transitive}, its autotopism group acts transitively
on all~$64$ cells.

For the small orders $n=4$, $5$, $6$, $7$, \ $9$, $10$, $11$ (different from $8$), there are only two $K_{3,3}$-free
$(n-1)\times n$ latin rectangles, with $n=5$ and $n=9$. 
\begin{figure}[ht]
$$
  \begin{array}{|ccccc|}
  \hline
0 & 3 & 4 & 2  & 1 \\
4 & 0 & 1 & 3  & 2 \\
2 & 4 & 0 & 1  & 3 \\
3 & 1 & 2 & 0  & 4 \\ \hline
                     \hline
\color{grey}1 & \color{grey}2 & \color{grey}3 & \color{grey}4  & \color{grey} 0 \\ \hline
   \end{array}
   \quad
    \begin{array}{|ccccccccc|}
  \hline
            0 & \color{red}     1 & \color{green!60!black}   4 & \color{green!60!black}   7 & \color{red}     2 & \color{blue}    8 & \color{brown}  3 & \color{blue}    6 &          5 \\
   \color{blue}    7 &          0 & \color{red}     2 & \color{green!60!black}   5 & \color{green!60!black}   8 & \color{red}     3 & \color{blue}    1 & \color{brown}  4 &          6 \\
   \color{brown}  5 & \color{blue}    8 &          0 & \color{red}     3 & \color{green!60!black}   6 & \color{green!60!black}   1 & \color{red}     4 & \color{blue}    2 &          7 \\
   \color{blue}    3 & \color{brown}  6 & \color{blue}    1 &          0 & \color{red}     4 & \color{green!60!black}   7 & \color{green!60!black}   2 & \color{red}     5 &          8 \\
   \color{red}     6 & \color{blue}    4 & \color{brown}  7 & \color{blue}    2 &          0 & \color{red}     5 & \color{green!60!black}   8 & \color{green!60!black}   3 &          1 \\
   \color{green!60!black}   4 & \color{red}     7 & \color{blue}    5 & \color{brown}  8 & \color{blue}    3 &          0 & \color{red}     6 & \color{green!60!black}   1 &          2 \\
   \color{green!60!black}   2 & \color{green!60!black}   5 & \color{red}     8 & \color{blue}    6 & \color{brown}  1 & \color{blue}    4 &          0 & \color{red}     7 &          3 \\
   \color{red}     8 & \color{green!60!black}   3 & \color{green!60!black}   6 & \color{red}     1 & \color{blue}    7 & \color{brown}  2 & \color{blue}    5 &          0 &          4 \\ \hline  \hline
           \color{grey} 1 &         \color{grey} 2 &         \color{grey} 3 &         \color{grey} 4 &         \color{grey} 5 &         \color{grey} 6 &         \color{grey} 7 &        \color{grey}  8 &         \color{grey} 0 \\   \hline
   \end{array}
 $$
\caption{The $K_{3,3}$-free latin rectangles $4\times 5$ and $8\times 9$, with their continuations to latin squares.}
\label{t:9}
\end{figure}
\begin{figure}[ht]
\small
$$
\def\0{\mathrm{A}}
\def\1{\mathrm{B}}
  \begin{array}{|@{\,}*{12}{c@{\,}}|}
  \hline
 0 & 1 & 2 &  3 & 4 & 5 &  6 & 7 & 8 &  9 &\0 &\1   \\
 1 & 2 & 0 &  4 & 5 & 3 &  7 & 8 & 6 & \0 &\1 & 9   \\
 2 & 6 & 3 &  5 & 9 & 0 &  8 &\0 & 1 & \1 & 7 & 4   \\
 3 & 5 & 4 &  0 & 2 & 1 & \0 & 9 &\1 &  7 & 6 & 8   \\
 4 & 3 & 5 &  1 & 0 & 2 &  9 &\1 &\0 &  6 & 8 & 7   \\
 5 & 0 & 8 & \0 & 1 & 4 & \1 & 2 & 7 &  3 & 9 & 6   \\
 6 & 8 & 7 & \1 &\0 & 9 &  2 & 1 & 0 &  4 & 3 & 5   \\
 7 & 4 & 1 &  2 & 3 &\1 &  0 & 6 & 9 &  8 & 5 &\0   \\
 8 & 7 & 6 &  9 &\1 &\0 &  1 & 0 & 2 &  5 & 4 & 3   \\
 9 &\1 &\0 &  8 & 7 & 6 &  5 & 4 & 3 &  1 & 0 & 2   \\
\0 & 9 &\1 &  7 & 6 & 8 &  4 & 3 & 5 &  2 & 1 & 0   \\ \hline
   \end{array}
   \qquad
  \begin{array}{|@{\,}*{12}{c@{\,}}|}
  \hline
 0 & 1 & 2 & 3 &  4 & 5 &  6 & 7 &  8 & 9 &\0 &\1  \\
 1 & 2 & 3 & 0 &  5 & 4 &  7 & 6 &  9 &\0 &\1 & 8  \\
 2 & 9 & 1 & 8 & \0 & 3 & \1 & 4 &  6 & 0 & 7 & 5  \\
 3 & 7 & 0 &\1 &  6 & 2 &  8 & 5 & \0 & 1 & 9 & 4  \\
 4 & 6 & 5 &\0 &  7 & 1 &  9 & 0 & \1 & 2 & 8 & 3  \\
 5 & 8 & 4 & 9 & \1 & 0 & \0 & 1 &  7 & 3 & 6 & 2  \\
 6 & 0 & 7 & 1 &  3 & 8 &  2 & 9 &  4 &\1 & 5 &\0  \\
 7 & 5 & 6 & 4 &  2 &\1 &  3 &\0 &  1 & 8 & 0 & 9  \\
 8 &\1 &\0 & 7 &  9 & 6 &  5 & 3 &  0 & 4 & 2 & 1  \\
 9 & 4 & 8 & 5 &  1 &\0 &  0 &\1 &  2 & 6 & 3 & 7  \\
\0 & 3 &\1 & 2 &  0 & 9 &  1 & 8 &  5 & 7 & 4 & 6  \\ \hline
   \end{array}
   \qquad
  \begin{array}{|@{\,}*{12}{c@{\,}}|}
  \hline
 0 & 1 & 2 & 3 & 4 & 5 & 6 & 7 & 8 & 9 &\0 &\1 \\
 1 & 2 & 3 & 4 & 5 & 6 & 7 & 8 & 0 &\0 &\1 & 9 \\
 2 & 8 & 7 & 5 &\0 & 9 & 0 & 4 & 3 &\1 & 6 & 1 \\
 3 & 0 & 8 & 1 & 9 & 4 &\1 &\0 & 7 & 6 & 5 & 2 \\
 4 & 3 &\1 & 8 & 2 & 1 &\0 & 9 & 6 & 5 & 0 & 7 \\
 5 & 4 & 0 & 9 & 6 & 3 & 2 & 1 &\1 & 8 & 7 &\0 \\
 6 &\1 &\0 & 0 & 8 & 7 & 4 & 3 & 2 & 1 & 9 & 5 \\
 7 &\0 & 9 & 2 &\1 & 0 & 8 & 5 & 1 & 4 & 3 & 6 \\
 8 & 9 & 4 &\1 & 7 &\0 & 3 & 6 & 5 & 2 & 1 & 0 \\ \hline
   \end{array}
 $$
\caption{Two $K_{3,3}$-free latin $11\times 12$ rectangles and one $K_{3,3}$-free latin $9\times 12$ rectangle without intercalates.}
\label{t:12}
\end{figure}
Both rectangles (see Fig.~\ref{t:9}) have some nice symmetries (isotopisms, paratopisms)
and
may potentially generalize to higher orders.
In each case,
the last row shown in Fig.~\ref{t:9} does not
belong to the rectangle;
it shows the unique completion to a latin square.
This square is not $K_{3,3}$-free, but all $K_{3,3}$-patterns ($12$ patterns and $60$ patterns, respectively, for orders $5$ and~$9$) necessarily
involve elements from the last row.
Additionally, with a nonexhaustive search, we found two $K_{3,3}$-free latin $11\times 12$ rectangles, see Fig.~\ref{t:12};
each of them has only $2$ autotopisms and $4$ paratopisms (actually the search approach was similar, but at each step we kept only rectangles with a nontrivial paratopism group).

We also performed a search for $K_{3,3}$-free latin rectangles that avoid intercalates
(an intercalate is a latin $2\times 2$ subsquare, which is one of the most famous substructures
in latin squares; for latin squares of order larger than~$4$, intercalates are equivalent to size-$4$ maximal cliques in the latin square graph). The results of their classification are shown in Table~\ref{t:inter}; the largest such rectangle we found is shown in Fig.~\ref{t:12}.

\begin{table}
 $$
\begin{array}{|c|*{10}{@{\ }c@{\ }|}}\hline
 m
   & n{:}\ 3 & 4 & 5 & 6 & 7 & 8 & 9 & 10 & 11 & 12 \\ \hline
3 & 0 & 0 & 1 & 1 & 4 & 17 & 121 & 963 & 9982 & 118923 \\[-0.9ex]
 &\scriptstyle( 0 )&\scriptstyle( 0 )&\scriptstyle( 1 )&\scriptstyle( 1 )&\scriptstyle( 6 )&\scriptstyle( 23 )&\scriptstyle( 199 )&\scriptstyle( 1766 )&\scriptstyle( 19278 )&\scriptstyle( 234764 )\\
4 & & 0 & 1 & 1 & 2 & 58 & 2577 & 194440 & 23133482 & 3424130183 \\[-0.9ex]
 & &\scriptstyle( 0 )&\scriptstyle( 1 )&\scriptstyle( 1 )&\scriptstyle( 2 )&\scriptstyle( 90 )&\scriptstyle( 4875 )&\scriptstyle( 386112 )&\scriptstyle( 46232421 )&\scriptstyle( 6847754337 )\\
5 & & & 0 & 0 & 0 & 11 & 1429 & 406986 & 377075287 & 574834480306 \\[-0.9ex]
 & & &\scriptstyle( 0 )&\scriptstyle( 0 )&\scriptstyle( 0 )&\scriptstyle( 16 )&\scriptstyle( 2713 )&\scriptstyle( 810190 )&\scriptstyle( 754037031 )&\scriptstyle( 1149663169917 )\\
6 & & & & 0 & 0 & 0 & 23 & 9632 & 31162840 & 297568453985 \\[-0.9ex]
 & & & &\scriptstyle( 0 )&\scriptstyle( 0 )&\scriptstyle( 0 )&\scriptstyle( 42 )&\scriptstyle( 19004 )&\scriptstyle( 62313657 )&\scriptstyle( 595135150850 )\\
7 & & & & & 0 & 0 & 3 & 33 & 52506 & 761716109 \\[-0.9ex]
 & & & & &\scriptstyle( 0 )&\scriptstyle( 0 )&\scriptstyle( 5 )&\scriptstyle( 59 )&\scriptstyle( 104658 )&\scriptstyle( 1523339709 )\\
8 & & & & & & 0 & 0 & 0 & 11 & 51871 \\[-0.9ex]
 & & & & & &\scriptstyle( 0 )&\scriptstyle( 0 )&\scriptstyle( 0 )&\scriptstyle( 19 )&\scriptstyle( 103458 )\\
9 & & & & & & & 0 & 0 & 0 & 1 \\[-0.9ex]
 & & & & & & &\scriptstyle( 0 )&\scriptstyle( 0 )&\scriptstyle( 0 )&\scriptstyle( 2 )\\\hline
\end{array}
$$
\caption{The number of main classes (isotopy classes) of $K_{3,3}$-free $m\times n$ latin rectangles without intercalates
($2$-by-$2$ latin subsquares)}
\label{t:inter}
\end{table}

\section
[A K(3,3)-free latin square of order 16]
{A $K_{3,3}$-free latin square of order $16$}
\label{s:16}
We were able to construct a $K_{3,3}$-free 
$16 \times 16$ latin square
in the following way,
based on the structure observed in the $8 \times 8$ case.

\begin{figure}[ht]
$$
\newcommand\0{\boldsymbol 0}
\newcommand\1{\boldsymbol 1}
\newcommand\2{\boldsymbol 2}
\newcommand\3{\boldsymbol 3}
\newcommand\4{\boldsymbol 4}
\newcommand\5{\boldsymbol 5}
\newcommand\6{\boldsymbol 6}
\newcommand\7{\boldsymbol 7}
\begin{array}{|@{\ }c@{\ }|@{\ }c@{\ }|@{\ }c@{\ }|@{\ }c@{\ }|}
 \hline
  0 & 2 & 4 & 6 \\\hline
  2 & 0 & 6 & 4 \\\hline
  4 & 6 & 0 & 2 \\\hline
  6 & 4 & 2 & 0 \\ \hline
 \end{array}\,,
\ \ 
\begin{array}{|@{\ }c@{\ }|@{\ }c@{\ }|@{\ }c@{\ }|@{\ }c@{\ }|}
 \hline
  1 & 7 & 3 & 5 \\\hline
  3 & 5 & 1 & 7 \\\hline
  5 & 3 & 7 & 1 \\\hline
  7 & 1 & 5 & 3 \\ \hline
 \end{array}
 \ \Longrightarrow
 \ 
 \begin{array}{|@{\ }c@{\ }c@{\ }|@{\ }c@{\ }c@{\ }|@{\ }c@{\ }c@{\ }|@{\ }c@{\ }c@{\ }|}
\hline
          \0 & 1 &  2 & 7 &  \4 & 3 &  6 & 5 \\
          1 & 0 &  7 & 2 &  3 & 4 &  5 & 6 \\  
\hline
          2 & 3 &  \0 & 5 &  \6 & 1 &  4 & 7 \\ 
          3 & 2 &  5 & 0 &  1 & 6 &  7 & 4 \\ 
\hline
          4 & 5 &  6 & 3 &  0 & 7 &  2 & 1 \\  
          5 & 4 &  3 & 6 &  7 & 0 &  1 & 2 \\
\hline                        
          \6 & 7 &  \4 & 1 &  2 & 5 &  0 & 3 \\ 
          7 & 6 &  1 & 4 &  5 & 2 &  3 & 0 \\ \hline
   \end{array}\qquad 
     \begin{array}{|@{\ }c@{\ }|@{\ }c@{\ }c@{\ }|@{\ }c@{\ }c@{\ }|@{\ }c@{\ }c@{\ }|}
\hline
          \gry 1 \ \gry\0 &  2 & 7 & \4 & 3 &  6 & 5 \\
          \gry\0 \ \gry 1 &  7 & 2 &  3 &\4 &  5 & 6 \\  
\hline
          2 \ 3 & \0 & 5 & \6 & 1 &  4 & 7 \\ 
          3 \ 2 &  5 &\0 &  1 &\6 &  7 & 4 \\ 
\hline
          4 \ 5 &  6 & 3 &  0 & 7 &  2 & 1 \\  
          5 \ 4 &  3 & 6 &  7 & 0 &  1 & 2 \\
\hline                        
         \6 \ 7 & \4 & 1 &  2 & 5 &  0 & 3 \\ 
          7\ \6 &  1 &\4 &  5 & 2 &  3 & 0 \\ \hline
   \end{array}
$$
\caption{An example of the $0$-combination and a switched combination of two orthogonal latin $4\times 4$ squares, where switching destroys a $K_{3,3}$-pattern  }
\label{t:ort4}
\end{figure}

Assume we have two $n\times n$ latin squares
$A_0$ and~$A_1$
and a $\{0,1\}$-valued array~$S$ of size $n\times n$, called a
\emph{switching matrix}.
We assume that $A_0$ and~$A_1$ have the same sets of rows and columns $\{0,1,\ldots,n-1\}$ but disjoint sets of letters,
say,
$\{0,2,\ldots,2n-2\}$ and $\{1,3,\ldots,2n-1\}$.
We define a $2n\times 2n$ latin square~$L$, called a
\emph{switched combination} of~$A_0$
and~$A_1$ (if $S$ is all-zero, just a \emph{$0$-combination}),
in the following way: 
\def\div#1#2{{#1/\!\!/#2}}
\begin{equation}\label{eq:Lij}
 L(r,c)=
A_{r\oplus c\oplus S(\div{r}{2},\div{c}{2})}(\div{r}{2},\div{c}{2}),
\end{equation}
where $\oplus$ is the addition modulo~$2$ and $\div{}{}$ is the integer division.
The construction is illustrated in Fig.~\ref{t:ort4}
where the first square of order~$8$ is the combination of the two latin squares of order~$4$ with  the all-zero switching matrix, while the second
square of order~$8$ is the combination with the switching matrix that has only one~$1$
in the left top position.
For the combined latin square, we refer to the set of four cells of the form $(2i,2j)$, $(2i,2j+1)$, $(2i+1,2j)$, $(2i+1,2j+1)$ as a \emph{block}.
 
\begin{lemma}
 For a switched combination of two orthogonal latin squares, 
 a forbidden pattern cannot be located in fewer than~$6$ blocks.
 Conversely, a switched combination of two nonorthogonal latin squares
 always contains a forbidden pattern located in $3$ blocks.
\end{lemma}
\begin{proof}
 If a forbidden pattern is located in fewer than~$6$ blocks,
 then one block contains two elements of it. Those two elements 
 can be two cells with the same letter or two cells with different letters but in the same row or column.
 In both cases, it is easy to see that the pattern is actually located in three blocks
 $$
      \begin{array}{|@{\ }c@{\ }c@{\ }|@{\ }c@{\ }c@{\ }|}
\hline
          \cdot &  A &  B & \cdot  \\
          A &  \cdot &  C & \cdot  \\  
\hline
          \cdot & \cdot &  \cdot & \cdot  \\ 
          B & C &  \cdot & \cdot  \\
\hline
   \end{array}
$$
and there are two blocks that contain the same pair of letters. 
This contradicts the definition of orthogonality and proves the first claim.

The second claim is simple to see, as a switched combination of two nonorthogonal
latin squares contains two blocks with the same pair of letters and hence contains
a forbidden pattern as above.
\end{proof}

So, if we want to get a $K_{3,3}$-free switched combination, then the combined latin squares
must be orthogonal and in this case it suffices to avoid forbidden pattern with 
$6$ elements in $6$ distinct blocks.
In the first $8 \times 8$ latin square in Fig.~\ref{t:ort4}, 
a pattern is indicated by six bold letters.
In the second $8 \times 8$ latin square, we switched one of the six blocks
containing that pattern and see that those six blocks do not contain a forbidden pattern anymore.
\begin{lemma} {\rm(i)}
 If a switched combination of two latin squares contains a forbidden pattern in $6$ blocks,
 then the switching matrix~$S$ has an even number of ones in the positions corresponding
 to those $6$ blocks.
  {\rm(ii)} Moreover, any other switching matrix with an even number of ones in those six positions
 gives (by combining the same two latin squares) a latin square
 with a forbidden pattern in the corresponding six blocks.
\end{lemma}
\begin{proof}
By Proposition~\ref{p:six}, we can assume that the six cells in a forbidden pattern are
 $x=(r_1,c_2,l_3)$,  $x'=(r_2,c_3,l_1)$,  $x''=(r_3,c_1,l_2)$,
$y=(r_2,c_1,l_3)$, $y'=(r_3,c_2,l_1)$, $y''=(r_1,c_3,l_2)$.

 (i)  Consider the cells $x$ and $y$, which have the same letter~$l_3$.
 From~\eqref{eq:Lij}, we see that
 $ r_1\oplus c_2\oplus S(\div{r_1}{2},\div{c_2}{2})$ equals $ r_2\oplus c_1\oplus S(\div{r_2}{2},\div{c_1}{2} )$ (which is $0$ if $l_3$ is even and $1$ if $l_3$ is odd). So, we have
 \begin{equation}\label{eq:xy}
  r_1\oplus c_2 \oplus r_2\oplus c_1 = S(\div{r_1}{2},\div{c_2}{2})\oplus S(\div{r_2}{2},\div{c_1}{2}).
 \end{equation}
 Considering similarly $x'$ and $y'$ (with letter~$l_1$), then $x''$ and $y''$ (with letter~$l_2$), we get
 \begin{equation}\label{eq:xy'}
  r_2\oplus c_3 \oplus r_3\oplus c_2
  =
  S(\div{r_2}{2},\div{c_3}{2}) \oplus S(\div{r_3}{2},\div{c_2}{2})
 \end{equation}
 and
 \begin{equation}\label{eq:xy''}
  r_3\oplus c_1 \oplus r_1\oplus c_3
  =
  S(\div{r_3}{2},\div{c_1}{2}) \oplus S(\div{r_1}{2},\div{c_3}{2}).
 \end{equation}
 Summing the three equations above, we get zero in the left side (indeed, each value is included twice) and the sum of
 the six values of~$S$ in the right side, which proves that the number of nonzero values among this six is even.

 (ii) Now, assume we have another switching matrix $\tilde S$
 that switches an even number of blocks among the six blocks
 containing $x$, $x'$, $x''$, $y$, $y'$, $y''$.
 In the new combined latin square $\tilde L$,
 constructed as the $\tilde S$-switched combination of the same two latin squares of order~$n$,
 we consider the following six elements, expecting that they will form a forbidden pattern in the same six cells and with the same letters as $x$, $x'$, $x''$, $y$, $y'$, $y''$:
 \begin{itemize}
  \item
  $\tilde x=(\tilde r_1,\tilde c_2,l_3)$ is any
 of the two elements of $\tilde L$
 in the block containing~$x$ (in~$L$)
 that have the letter~$l_3$ (the letter of~$x$).
 \item
 $\tilde y'=(\tilde r_3,\tilde c_2,l_1)$ is the element of $\tilde L$ lying in the same block as~$y'$
 and in the same column $\tilde c_2$ as~$\tilde x$ and having
 the letter~$l_1$ (the letter of~$y'$).
 \item
 $\tilde x''=(\tilde r_3,\tilde c_1,l_2)$
 is the element of $\tilde L$ lying in the same block as~$x''$
 and the same row $\tilde r_3$ as~$\tilde y'$ and having
 the letter~$l_2$ (the letter of~$x'$).
 \item
 $\tilde y=(\tilde r_2,\tilde c_1,l_3)$ is the element of $\tilde L$ lying in the same block as~$y$
 and the same column $\tilde c_1$ as~$\tilde x''$ and having
 the letter~$l_3$ (the letter of~$y$).
 \item
 $\tilde x'=(\tilde r_2,\tilde c_3,l_1)$
 is the element of $\tilde L$ lying in the same block as~$x'$
 and the same row $\tilde r_3$ as~$\tilde y$ and having
 the letter~$l_1$ (the letter of~$x'$).
 \item Finally,
 $\tilde y''=(\tilde r_1,\tilde c_3,\tilde l_2)$
 is the element of~$\tilde L$
 in the same column~$\tilde c_3$ as~$\tilde x'$ and
 in the same row $\tilde r_1$ as~$\tilde x$.

 \end{itemize}
To see that these six elements of $\tilde L$
form a forbidden pattern, it remains to ensure
that $\tilde l_2 = l_2$. Assuming the contrary, we have that
$\tilde x''$ and $\tilde y''$ have different letters,
while $\tilde x$ and $\tilde y$ have equal letters,
and $\tilde x'$ and $\tilde y'$ have equal letters.
Similarly to~\eqref{eq:xy}--\eqref{eq:xy''}, from~\eqref{eq:Lij}
we derive
 \begin{eqnarray*}\label{eq:txy''}
  \tilde r_3\oplus \tilde c_1 \oplus \tilde r_1\oplus \tilde  c_3
  &\!\!\not =\!\!&
  \tilde S(\div{\tilde r_3}{2},\div{\tilde c_1}{2}) \oplus \tilde S(\div{\tilde r_1}{2},\div{\tilde c_3}{2}),
  \\
    \tilde r_1\oplus \tilde c_2 \oplus \tilde r_2\oplus \tilde c_1
    &\!\!=\!\!&
    \tilde S(\div{\tilde r_1}{2},\div{\tilde c_2}{2})\oplus \tilde S(\div{\tilde r_2}{2},\div{\tilde c_1}{2}),
    \\
  \tilde r_2\oplus \tilde c_3 \oplus \tilde r_3\oplus \tilde c_2
  &\!\!=\!\!&
  \tilde S(\div{\tilde r_2}{2},\div{\tilde c_3}{2}) \oplus \tilde S(\div{\tilde r_3}{2},\div{\tilde c_2}{2}).
 \end{eqnarray*}
 Summing these three (in)equations with~$\oplus$,
 we get $0$ in the left (each value repeats twice), $0$ in the right (even number of switches), and $\not=$ in between,
 a~contradiction. Hence, $\tilde l_2 = l_2$ and the six elements $\tilde x$, $\tilde x'$, $\tilde x''$, $\tilde y$, $\tilde y'$, $\tilde y''$ form a forbidden pattern.
\end{proof}

Now, the strategy is as follows. We take a pair of orthogonal latin squares $A_0$, $A_1$
of order~$n$
and consider the elements $s_{i,j}$, $i,j=0,\ldots,n-1$ 
of an $n\times n$ switching matrix as variables.
We construct the $0$-combination $L_0$ of~$A_0$ and~$A_1$ and locate all forbidden patterns.
For a forbidden pattern in blocks with coordinates 
$(i  ,j'')$,
$(i' ,j'')$,
$(i  ,j  )$,
$(i' ,j' )$,
$(i'',j  )$,
$(i'',j' )$,
we consider the equation
$$ 
s_{i  ,j''} +
s_{i' ,j''} +
s_{i  ,j  } +
s_{i' ,j' } +
s_{i'',j  } +
s_{i'',j' }  = 1
$$
over GF$(2)$. All such equations, for all forbidden patterns in~$L_0$,
form a system of linear equations. If the system has a solution,
then a $K_{3,3}$-free latin square is constructed as the combination 
             of~$A_0$ and~$A_1$ with switching matrix from the solution.
             
As was found in \cite{EganWanless:2016}, there are $11887$ pairs of orthogonal $8\times 8$
latin squares, up to the following equivalence:
for our approach, the pairs of orthogonal latin squares are essentially the same
if one pair can be obtained from the other by isotopy,
swapping the roles of rows and columns, and/or swapping the two orthogonal squares
(however, for example, changing the roles of columns and letters of the first square
can result in the essentially different combination).
  Representatives of the $11887$ equivalence classes
  with respect to this equivalence
  can be found from representatives of the 
  $2165$ main classes of pairs of orthogonal $8\times 8$
latin squares~\cite{MK:data}.
Finally, we have found that one (and only one of~$11887$) pair of orthogonal latin squares
yields a solvable system, see Fig.~\ref{t:888}. There are $2^{15}$ solutions
(Fig.~\ref{t:888} shows only one of them, the other can be obtained
by changing all the values in some columns and rows). They result
in isotopic $K_{3,3}$-free latin squares of order~$16$; one of them is shown in 
Fig.~\ref{t:16}.
The found latin square has only $32$ autotopisms (the full paratopism group has order~$64$),
as well as the corresponding $0$-combination. In particular, the paratopism group
does not act transitively on cells (equivalently, the square is not isotopic to a G-loop),
in contrast to the $K_{3,3}$-free latin squares of order~$8$.
Each of the two combined $8\times 8$
latin squares is highly symmetric:
the autotopism group acts transitively on the $64$ cells
(that is, we have a G-loop isotope) and has order
$64$ and $192$, respectively. However, the autotopism group
of the pair of orthogonal latin squares has order~$16$ only.
The first of the combined $8\times 8$ squares (Fig.~\ref{t:888})
is itself a switched combination of the $4\times 4$
squares from Fig.~\ref{t:8}; notably,
the switching matrix is the mod-$2$ sum of
the switching matrices for the two $K_{3,3}$-free
squares in Fig.~\ref{t:8}
(see the grey cells in Fig.~\ref{t:8} and Fig.~\ref{t:888}).
The structure of the second
square in Fig.~\ref{t:888} is less understood.

\begin{problem}
 What is so special in those two orthogonal latin squares (Fig.~\ref{t:888})?
 Can the construction of $K_{3,3}$-free (universally noncommutative) latin squares
 be generalized to higher orders?
\end{problem}

\begin{figure}[ht]
$$
\newcommand\0{\boldsymbol 0}
\newcommand\1{\boldsymbol 1}
\newcommand\2{\boldsymbol 2}
\newcommand\3{\boldsymbol 3}
\newcommand\4{\boldsymbol 4}
\newcommand\5{\boldsymbol 5}
\newcommand\6{\boldsymbol 6}
\newcommand\7{\boldsymbol 7}
\begin{array}{|@{\ }c@{\ \,}c@{\ }|@{\ }c@{\ \,}c@{\ }|@{\ }c@{\ \,}c@{\ }|@{\ }c@{\ \,}c@{\ }|}
\hline
  0 & 1 &\gry 3 &\gry 2 & 4 & 5 & 6 & 7 \\
  1 & 0 &\gry 2 &\gry 3 & 5 & 4 & 7 & 6 \\ \hline
  2 & 5 & 0 & 7 &\gry 1 &\gry 6 & 4 & 3 \\
  5 & 2 & 7 & 0 &\gry 6 &\gry 1 & 3 & 4 \\ \hline
  4 & 7 & 6 & 5 & 0 & 3 &\gry 1 &\gry 2 \\
  7 & 4 & 5 & 6 & 3 & 0 &\gry 2 &\gry 1 \\ \hline
 \gry 3 &\gry 6 & 4 & 1 & 2 & 7 & 0 & 5 \\
 \gry 6 &\gry 3 & 1 & 4 & 7 & 2 & 5 & 0 \\   \hline
 \end{array}
\qquad
\begin{array}{|@{\ }c@{\ \ }c@{\ \ }c@{\ \ }c@{\ \ }c@{\ \ }c@{\ \ }c@{\ \ }c@{\ }|} 
\hline
  0 & 2 & 5 & 6 & 7 & 4 & 1 & 3 \\ 
  1 & 3 & 4 & 7 & 6 & 5 & 0 & 2 \\ 
  2 & 0 & 7 & 4 & 5 & 6 & 3 & 1 \\ 
  3 & 1 & 6 & 5 & 4 & 7 & 2 & 0 \\              
  4 & 7 & 3 & 1 & 2 & 0 & 6 & 5 \\      
  5 & 6 & 2 & 0 & 3 & 1 & 7 & 4 \\       
  6 & 5 & 1 & 3 & 0 & 2 & 4 & 7 \\        
  7 & 4 & 0 & 2 & 1 & 3 & 5 & 6 \\   \hline
 \end{array}
\qquad
\begin{array}{|@{\ }c@{\ \ }c@{\ \ }c@{\ \ }c@{\ \ }c@{\ \ }c@{\ \ }c@{\ \ }c@{\ }|} 
\hline
  1 & 0 & 0 & 0 & 1 & 0 & 0 & 0 \\ 
  0 & 1 & 0 & 0 & 0 & 1 & 0 & 0 \\ 
  0 & 0 & 0 & 1 & 0 & 0 & 0 & 1 \\ 
  0 & 0 & 1 & 0 & 0 & 0 & 1 & 0 \\ 
  0 & 1 & 0 & 0 & 0 & 1 & 0 & 0 \\ 
  1 & 0 & 0 & 0 & 1 & 0 & 0 & 0 \\ 
  0 & 0 & 1 & 0 & 0 & 0 & 1 & 0 \\ 
  0 & 0 & 0 & 1 & 0 & 0 & 0 & 1 \\ \hline
 \end{array}
$$
\caption{Two orthogonal latin squares and a switching function that give a $K_{3,3}$-free latin square of order $16$ as a switched combination}
\label{t:888}
\end{figure}

\begin{figure}[ht]
$$
\def\0{\mathrm{A}}
\def\1{\mathrm{B}}
\def\2{\mathrm{C}}
\def\3{\mathrm{D}}
\def\4{\mathrm{E}}
\def\5{\mathrm{F}}
\begin{array}{|@{\ }c@{\ }c@{\ }|@{\ }c@{\ }c@{\ }|@{\ }c@{\ }c@{\ }|@{\ }c@{\ }c@{\ }|@{\ }c@{\ }c@{\ }|@{\ }c@{\ }c@{\ }|@{\ }c@{\ }c@{\ }|@{\ }c@{\ }c@{\ }|} 
\hline
 1 & 0 & 2 & 5 & 6 &\1 & 4 &\3 &\5 & 8 &\0 & 9 &\2 & 3 &\4 & 7\\
 0 & 1 & 5 & 2 &\1 & 6 &\3 & 4 & 8 &\5 & 9 &\0 & 3 &\2 & 7 &\4\\\hline
 2 & 3 & 7 & 0 & 4 & 9 & 6 &\5 &\0 &\3 &\1 & 8 &\4 & 1 &\2 & 5\\
 3 & 2 & 0 & 7 & 9 & 4 &\5 & 6 &\3 &\0 & 8 &\1 & 1 &\4 & 5 &\2\\\hline
 4 & 5 &\0 & 1 & 0 &\5 & 9 &\4 & 2 &\1 &\2 &\3 & 8 & 7 & 3 & 6\\
 5 & 4 & 1 &\0 &\5 & 0 &\4 & 9 &\1 & 2 &\3 &\2 & 7 & 8 & 6 & 3\\\hline
\0 & 7 & 4 & 3 &\3 &\4 & 0 &\1 &\2 & 9 & 2 &\5 & 5 & 6 & 8 & 1\\
 7 &\0 & 3 & 4 &\4 &\3 &\1 & 0 & 9 &\2 &\5 & 2 & 6 & 5 & 1 & 8\\\hline
 8 & 9 &\5 &\4 &\2 & 7 &\0 & 3 & 0 & 5 & 1 & 6 & 2 &\3 & 4 &\1\\
 9 & 8 &\4 &\5 & 7 &\2 & 3 &\0 & 5 & 0 & 6 & 1 &\3 & 2 &\1 & 4\\\hline
\1 &\4 & 8 &\3 &\0 & 5 &\2 & 1 & 7 & 6 & 0 & 3 & 4 &\5 & 2 & 9\\
\4 &\1 &\3 & 8 & 5 &\0 & 1 &\2 & 6 & 7 & 3 & 0 &\5 & 4 & 9 & 2\\\hline
 6 &\3 &\2 &\1 & 3 & 8 & 2 & 7 & 4 & 1 &\4 & 5 & 9 & 0 &\0 &\5\\
\3 & 6 &\1 &\2 & 8 & 3 & 7 & 2 & 1 & 4 & 5 &\4 & 0 & 9 &\5 &\0\\\hline
\2 &\5 & 6 & 9 & 2 & 1 & 5 & 8 &\4 & 3 & 4 & 7 &\0 &\1 &\3 & 0\\
\5 &\2 & 9 & 6 & 1 & 2 & 8 & 5 & 3 &\4 & 7 & 4 &\1 &\0 & 0 &\3\\\hline
 \end{array}
$$
\caption{A $K_{3,3}$-free (universally noncommutative) latin square of order $16$}
\label{t:16}
\end{figure}
\section{\texorpdfstring{$K_{4,4}$}{K44}-free pairs of orthogonal latin squares}\label{s:ort}
The problem of the existence of $K_{3,3}$-free latin squares
can be naturally generalized to sets of mutually orthogonal latin squares.
It is known that a set of $k$, $1\le k \le n-2$, mutually orthogonal latin squares (MOLS) of order~$n$
corresponds to a
strongly regular
graph.
According to \cite[Theorem~3]{KMP:16:trades}, an eigenfunction of this graph corresponding to the eigenvalue $-k-2$ has the minimum number $2(k+2)$ of nonzeros only if these nonzeros
induce a subgraph isomorphic to $K_{k+2,k+2}$.
The study of $K_{k+2,k+2}$-free collections of $k$ MOLS is nontrivial even in the linear case,
when $n$ is a prime power and all latin squares are linear functions
over GF$(n)$:
among the linear pairs $(G,L)$ of orthogonal latin squares over GF$(n)$, where $n\ge 7$ is an odd prime power,
both $K_{4,4}$-free and non-$K_{4,4}$-free pairs exist, see Theorem~\ref{th:K44nK44} below.

We first adapt the terminology to pairs of orthogonal latin squares.
Such a pair $(A,B)$ of order~$n$ will also be treated as the set $\{(r,c,G(r,c),L(r,c)):\ c\in C, r\in R \}$ of ordered quadruples;
the graph $\Gamma(A,B)$ of this pair has this set as the vertex set, two quadruples being adjacent
if they coincide in one position.
A set of $8$ elements of a pair $(A,B)$ of orthogonal latin squares is called
a $K_{4,4}$-pattern if it induces a subgraph of $\Gamma(A,B)$ isomorphic to $K_{4,4}$. If no such patterns exist, the pair $(A,B)$ is called  \emph{$K_{4,4}$-free}.

Two pairs of orthogonal latin squares
$(A,B)\subset {R} \times {C} \times {G} \times {L}$ and
$(A',B')\subset {R'} \times {C'} \times {G'} \times {L'}$ are \emph{isotopic} if there are bijections $\rho:R\to R'$, $\gamma:C\to C'$, $\delta:G\to G'$, and $\lambda:L\to L'$
such that
$(A',B') = \{(\rho r,\gamma c, \delta g, \lambda l):\ (r,c,g,l) \in (A,B)\}$.
If
$(A',B') = \{\pi(\rho r,\gamma c, \delta g, \lambda l):\ (r,c,g,l) \in (A,B)\}$,
where $\pi$ is one of $4!$ permutations of coordinates,
then $(A,B)$ and $(A',B')$ are called \emph{paratopic}.


In Fig.~\ref{t:K44}, we can see examples of $K_{4,4}$-patterns in linear pairs of orthogonal latin squares of order~$4$ and~$5$.
Note that there are two types of $K_{4,4}$ subgraphs in a graph of a pair of orthogonal latin squares.
To distinguish these two types, we define the color
of an edge as follows: in the graph of two orthogonal latin squares $G$ and $L$, a pair of adjacent vertices
$(r,c,g,l)$ and $(r',c',g',l')$
is assigned color $1$, $2$, $3$, or~$4$ if
$r=r'$, $c=c'$, $g=g'$, or $l=l'$, respectively.
In a $K_{4,4}$-subgraph, the edges of each color
form a $1$-factor;
the edges of any two colors
form a $2$-factor, which is an $8$-cycle or the union
of two $4$-cycles. We define the type of a $K_{4,4}$-subgraph as the number of two-color $8$-cycles in it.
\begin{lemma}
 There are only two possible types, type~$0$  and type~$4$.
\end{lemma}
Examples of type-0 and type-4 patterns can be found
in Fig.~\ref{t:K44} (the $4\times 4$ square and the two $5\times 5$ squares, respectively).
\begin{proof}
Let $P_0$ and $P_1$ be the bipartite parts of~$K_{4,4}$.
Straightforwardly,
the set $(r,c,l)$, where $r\in P_0$, $c\in P_1$,
and $l$ is the color of the edge $\{r,c\}$,
is a latin square of order~$4$.
The two isotopy classes
of latin squares of order~$4$
correspond to the two types of colorings, which can be checked directly.
\end{proof}
\begin{figure}
 $$
  \begin{array}{|c@{\ }c@{\ }c@{\ }c|}
  \hline
{\color{green!60!black}\boldsymbol\alpha{ }\boldsymbol{A}} & {\color{blue}\boldsymbol\beta{ }\boldsymbol{C}} & {\gamma{ }D} & {\delta{ }B} \\
{\beta{ }B} & {\alpha{ }D} & {\color{green!60!black}\boldsymbol\delta{ }\boldsymbol{C}} & {\color{blue}\boldsymbol\gamma{ }\boldsymbol{A}} \\
{\gamma{ }C} & {\delta{ }A} & {\color{blue}\boldsymbol\alpha{ }\boldsymbol{B}} & {\color{green!60!black}\boldsymbol\beta{ }\boldsymbol{D}} \\
{\color{blue}\boldsymbol\delta{ }\boldsymbol{D}} & {\color{green!60!black}\boldsymbol\gamma{ }\boldsymbol{B}} & {\beta{ }A} & {\alpha{ }C} \\ \hline
   \end{array}
   \quad
   \begin{array}{|c@{\ }c@{\ }c@{\ }c@{\ }c|}
   \hline
{\alpha{ }A} & {\beta{ }C} & {\gamma{ }E} & {\delta{ }B} & {\varepsilon{ }D} \\
{\beta{ }B} & {\color{green!60!black}\boldsymbol\gamma{ }\boldsymbol{D}} & {\delta{ }A} & {\color{blue}\boldsymbol\varepsilon{ }\boldsymbol{C}} & {\alpha{ }E} \\
{\gamma{ }C} & {\color{blue}\boldsymbol\delta{ }\boldsymbol{E}} & {\color{green!60!black}\boldsymbol\varepsilon{ }\boldsymbol{B}} & {\alpha{ }D} & {\beta{ }A} \\
{\delta{ }D} & {\varepsilon{ }A} & {\alpha{ }C} & {\color{green!60!black}\boldsymbol\beta{ }\boldsymbol{E}} & {\color{blue}\boldsymbol\gamma{ }\boldsymbol{B}} \\
{\varepsilon{ }E} & {\alpha{ }B} & {\color{blue}\boldsymbol\beta{ }\boldsymbol{D}} & {\gamma{ }A} & {\color{green!60!black}\boldsymbol\delta{ }\boldsymbol{C}} \\ \hline
    \end{array}
    \quad
  \begin{array}{|c@{\ }c@{\ }c@{\ }c@{\ }c|}
  \hline
{\alpha{ }{A}} & {\color{green!60!black}\boldsymbol\beta{ }\boldsymbol{E}} & {\color{blue}\boldsymbol\gamma{ }\boldsymbol{D}} & {\delta{ }{C}} & {\varepsilon{ }{B}} \\
{\color{blue}\boldsymbol\beta{ }\boldsymbol{B}} & {\gamma{ }{A}} & {\delta{ }{E}} & {\color{green!60!black}\boldsymbol\varepsilon{ }\boldsymbol{D}} & {\alpha{ }{C}} \\
{\color{green!60!black}\boldsymbol\gamma{ }\boldsymbol{C}} & {\delta{ }{B}} & {\varepsilon{ }{A}} & {\color{blue}\boldsymbol\alpha{ }\boldsymbol{E}} & {\beta{ }{D}} \\
{\delta{ }{D}} & {\color{blue}\boldsymbol\varepsilon{ }\boldsymbol{C}} & {\color{green!60!black}\boldsymbol\alpha{ }\boldsymbol{B}} & {\beta{ }{A}} & {\gamma{ }{E}} \\
{\varepsilon{ }{E}} & {\alpha{ }{D}} & {\beta{ }{C}} & {\gamma{ }{B}} & {\delta{ }{A}} \\ \hline
   \end{array}
$$
\caption{In the (unique, up to equivalence) pair $(G,L)$ of orthogonal latin squares of order~$4$, every pair of transversals induces $K_{4,4}$. In the (unique, up to equivalence) pair $(G,L)$ of orthogonal latin squares of order~$5$, the difference between two transversals through the same element induces $K_{4,4}$. In the tables, the values of $G$ and $L$ are indicated by Greek and Latin letters, respectively (pairs of orthogonal latin squares are also called Graeco-Latin squares).
}\label{t:K44}
\end{figure}

\begin{theorem}\label{th:K44nK44}
 Let $G,L_\kappa:\mathrm{GF}(q)^2\to \mathrm{GF}(q)$
 be two orthogonal latin squares of order~$q$,
 where $q$ is a prime power,
 \begin{equation*}\label{eq:GLk}
   G(r,c)=r+c,\quad
 L_\kappa(r,c)=r+\kappa c,\qquad \kappa \in \mathrm{GF}(q)\setminus \{0,1\}.
 \end{equation*}
 \begin{itemize}
  \item[\rm(i)] The pair $(G,L_\kappa)$ contains a forbidden pattern
 of type~$4$ if and only if $q$ is odd and $\kappa\in\{2,1/2,-1\}$.
  \item[\rm(ii)] The pair $(G,L_\kappa)$ contains a forbidden pattern
 of type~$0$ if and only if $q$ is even.
 \end{itemize}
\end{theorem}

\begin{proof} (i) \emph{Only if.}
We first consider the case when
\begin{itemize}
 \item[(a)] in the $K_{4,4}$-subgraph,
the edges of colors $1$ and~$2$ form two $4$-cycles and the edges of colors $3$ and~$4$ form two $4$-cycles in the $K_{4,4}$-subgraph;
 \item[(b)] the edges corresponding to any other pair of colors ($1$ and~$3$, $1$ and~$4$, $2$ and~$3$, $2$ and~$4$) form an $8$-cycle in the $K_{4,4}$-subgraph.
\end{itemize}
We are going to show that the pattern is similar to the one in the right square in Fig.~\ref{t:K44}, which illustrates the considerations below.
 Let the two $4$-cycles formed by edges of colors
 $1$ and $2$ be $v_{00}{-}v_{10}{-}v_{11}{-}v_{01}{-}v_{00}$ and
 $v_{22}{-}v_{32}{-}v_{33}{-}v_{23}{-}v_{22}$, where
$$
 \begin{array}{c@{}cc@{}c}
 &v_{22}=(r_2,c_2,\beta',E'), & v_{23}=(r_2,c_3,\gamma',D'), & \\
 v_{00}=(r_0,c_0,\beta,B), &&&
 v_{01}=(r_0,c_1,\varepsilon,D), \\
 v_{10}=(r_1,c_0,\gamma,C), &&&
 v_{11}=(r_1,c_1,\alpha,E), \\
 &v_{32}=(r_3,c_2,\varepsilon',C'), & v_{33}=(r_3,c_3,\alpha',B'). & \\
 \end{array}
$$
Obviously,
 $r_0$, $r_1$, $r_2$, $r_3$ are pairwise different;
 $c_0$, $c_1$, $c_2$, $c_3$ are pairwise different;
 $\beta$, $\gamma$, $\varepsilon$, $\alpha$ are pairwise different
 and $\{\beta$, $\gamma$, $\varepsilon$, $\alpha\}=\{\beta'$, $\varepsilon'$, $\gamma'$, $\alpha'\}$;
 $B$, $C$, $D$, $E$ are pairwise different
 and $\{B$, $C$, $D$, $E\}=\{E'$, $C'$, $D'$, $B'\}$.

 Without loss of generality, we assume $\beta=\beta'$,
 which means that the bipartite parts of $K_{4,4}$ are
 $\{v_{00},v_{11},v_{23},v_{32}\}$ and $\{v_{01},v_{10},v_{22},v_{33}\}$. Now, if $\gamma=\varepsilon'$, then we have a $4$-cycle
 $v_{00}{-}v_{10}{-}v_{32}{-}v_{22}{-}v_{00}$ with colors $2$ and~$3$ of the edges, which contradicts condition~(b) above. Hence,
 $\gamma\not=\varepsilon'$ and $C=C'$. Similarly, considering the $4$-cycles
$v_{00}{-}v_{01}{-}v_{23}{-}v_{22}{-}v_{00}$ and
$v_{01}{-}v_{11}{-}v_{33}{-}v_{23}{-}v_{01}$,
we get $D=D'$ and $\alpha=\alpha'$.

Now, since $\beta=\beta'$, $\gamma\not=\varepsilon'$, $\alpha=\alpha'$,  and
$\{\beta$, $\gamma$, $\varepsilon$, $\alpha\}=\{\beta'$, $\varepsilon'$, $\gamma'$, $\alpha'\}$, we have also $\gamma=\gamma'$
and $\varepsilon=\varepsilon'$. Similarly, $B=B'$ and $E=E'$.

From $G(r_0,c_0)=G(r_2,c_2)$ $(=\beta)$, we have
$r_0+c_0 = r_2+c_2$; from $G(r_1,c_0)=G(r_2,c_3)$ $(=\gamma)$, we have
$r_1+c_0 = r_2+c_3$. We derive
\begin{equation}\label{eq:1032}
 r_0-r_1 = c_2-c_3.
\end{equation}
On the other hand,
from $L_\kappa(r_0,c_0)=L_\kappa(r_3,c_3)$ $(=B)$
we have
$r_0+\kappa c_0 = r_3+\kappa c_3$;
from $L_\kappa(r_1,c_0)=L_\kappa(r_3,c_2)$ $(=C)$
we have
$r_1+\kappa c_0 = r_3+\kappa c_2$.
We derive
\begin{equation}\label{eq:1023}
 r_0-r_1 = \kappa (c_3-c_2).
\end{equation}
From \eqref{eq:1032} and \eqref{eq:1023}, we find $\kappa = -1$.

Now assume that, instead of (a) and (b), we have that the pairs of colors inducing two $4$-cycles in the $K_{4,4}$-subgraph
are $(1,3)$ and $(2,4)$. Since
$$
\left\{\begin{array}{ll}
 g=r+c\\
 l=r+\kappa c
\end{array}\right.
\quad\mbox{is equivalent to }
\left\{\begin{array}{ll}
 r=l+(-\kappa)c\\
 g=l+(1-\kappa) c
\end{array}\right. ,
$$
the pair $(G,L_\kappa)$ is paratopic to
$(G,L_{\kappa'})$ with $\kappa'=\frac{1-\kappa}{-\kappa}=1-\frac{1}{\kappa}$, where $(G,L_{\kappa'})$ satisfies (a) and~(b). We conclude $\kappa'=-1$ and $\kappa=2$.

Similarly, the remaining case ($4$-cycles corresponding to the pairs of colors $(1,4)$ and $(2,3)$) leads to $\kappa=1/2$. It remains to note that in all the cases, even~$q$ leads to a contradiction ($\kappa=-1 \equiv 1 \bmod 2$ and $\kappa=2 \equiv 0 \bmod 2$ contradict the definition of orthogonal latin squares, $\kappa=1/2$ does not exist).

(i) \emph{If.}
In the case $\kappa=-1$, the following eight vertices induce $K_{4,4}$:
 $$
 \begin{array}{cccc}
 &(0,1,1,-1), &
  (0,1{+}x,1{+}x,-1{-}x), & \\
 \makebox[1mm][c]{$(1,0,1,1)$,} &&&
 \makebox[1mm][c]{$(1,2{+}x,3{+}x,-1{-}x),$} \\
 \makebox[1mm][c]{$(1{+}x,0,1{+}x,1{+}x)$,} &&&
 \makebox[1mm][c]{$(1{+}x,2{+}x,3{+}2x,-1),$} \\
 & (2{+}x,1,3{+}x,1{+}x), &
 (2{+}x,1{+}x,3{+}2x,1), &
 \end{array}
$$ where $0$, $1$, $1+x$, $2+x$ are mutually different (for example, $x=1$ if the prime divisor of~$q$ is larger than~$3$ and $x\in\mathrm{GF}(q)\backslash\{0,1,2\}$ if $3$ divides~$q$).
From this example, using the paratopisms described at the end of the ``only if'' part of the proof, we obtain examples for $\kappa=2$ and $\kappa=1/2$.

(ii) \emph{Only if.}
Instead of (a) and (b), we now have
\begin{itemize}
 \item[(c)]
the edges of any two colors from $1$, $2$, $3$, $4$ form two $4$-cycles in the $K_{4,4}$-subgraph.
\end{itemize}
Defining $v_{00}$, \ldots, $v_{33}$ as before,
we now derive from~(c)
that $\varepsilon'=\gamma$,
$\gamma'=\varepsilon$,
$\alpha'=\alpha$, $B'=B$,
$E'=E$, $C'=D$, $D'=C$.
From $G(r_0,c_0)=G(r_2,c_2)$
and $G(r_1,c_0)=G(r_3,c_2)$,
we have $r_3-r_1=r_2-r_0$.
From $L(r_0,c_0)=L(r_3,c_3)$
and $L(r_1,c_0)=L(r_2,c_3)$,
we have $r_3-r_0 = r_2-r_1$.
In summary, we get $r_0-r_1=r_1-r_0$,
which can only be possible if $q$ is even.

(ii) \emph{If.} For $q$ a power of~$2$,
eight elements of $(G,L_\kappa)$ inducing $K_{4,4}$
are represented as columns of the following array:
$$
\begin{array}{r||c|c|c|c||c|c|c|c|}
            & v_{00} & v_{11}          & v_{23}      & v_{32}          & v_{01}          & v_{10}     & v_{22}   & v_{33}          \\\hline
 r          & 0      & 1+\kappa^2      & 1           & \kappa^2        & 0               & 1+\kappa^2 & 1        & \kappa^2        \\
 c          & 0      & 1+\kappa        & \kappa      & 1               & 1+\kappa        & 0          & 1        & \kappa          \\
 r+c        & 0      & \kappa+\kappa^2 & 1+\kappa    & 1+\kappa^2      & 1+\kappa        & 1+\kappa^2 & 0        & \kappa+\kappa^2 \\
 r+\kappa c & 0      & 1+\kappa        & 1+\kappa^2  & \kappa+\kappa^2 & \kappa+\kappa^2 & 1+\kappa^2 & 1+\kappa & 0               \\\hline
\end{array}
$$

\end{proof}

\begin{corollary}\label{cor:K44}
For every odd prime power $q\ge 7$, there exists a linear $K_{4,4}$-free pair of orthogonal latin squares. For every prime power $q\ge 4$, there exists a linear pair of orthogonal latin squares containing a $K_{4,4}$-pattern.
\end{corollary}

\section{Conclusion}
As we see,
$K_{3,3}$-free latin squares are very rare objects,
at least for small orders.
Apart from the two known examples of order~$8$,
we have constructed one of order~$16$, and we have also found
two very symmetric $K_{3,3}$-free latin $(n-1)\times n$
rectangles, $n=5$ and~$9$ (and two such rectangles with $n=12$, with few symmetries). The generalization of such objects
to higher orders remains an open problem.

The situation with $K_{4,4}$-patterns in pairs
of orthogonal latin squares seems to be different.
We have proved the existence of
$K_{4,4}$-free pairs of orthogonal latin squares
for each odd prime power order larger than~$5$.
Generalizations to other orders and to systems of more than two
orthogonal latin squares would be very interesting.

\appendix
\section*{Appendix: proofs of Propositions \ref{p:six} and \ref{p:commut}}\setcounter{proposition}{0}
Here we give proofs of two propositions found in the introduction.
\begin{proposition}\label{p:six+}
Six distinct elements $x$, $x'$, $x''$, $y$, $y'$, $y''$ of a latin rectangle $S$ induce a complete bipartite
subgraph of \,$\Gamma(S)$ with parts $\{x, x', x''\}$ and $\{y, y', y''\}$ if and only if
\begin{align}\label{eq:123}
 \{x,x',x''\}&=\{(r_1,c_2,l_3),(r_2,c_3,l_1),(r_3,c_1,l_2)\}
 \mbox{ and }\\ \label{eq:213}
 \{y, y', y''\}&=\{(r_1,c_3,l_2), (r_2,c_1,l_3), (r_3,c_2,l_1)\}
 \end{align}
for some $r_1$, $r_2$, $r_3$, $c_1$, $c_2$, $c_3$, $l_1$, $l_2$, $l_3$.
\end{proposition}
\begin{proof}
Let $x$, $x'$, $x''$, $y$, $y'$, $y''$
be six distinct elements of a latin rectangle~$S$.

 We first assume that $\{x, x', x''\}$ and $\{y, y', y''\}$ satisfy \eqref{eq:123}--\eqref{eq:213}. We note that
 $r_1\ne r_2$ because $(r_1,c_2,l_3)$
 and $(r_2,c_1,l_3)$ cannot be distinct elements of a latin rectangle otherwise.
 Similarly, $r_2\ne r_3\ne r_1$,
 $c_1\ne c_2\ne c_3\ne c_1$, and
 $l_1\ne l_2\ne l_3\ne l_1$. Now, we see that
 every two elements in $\{x,x',x''\}$ or in $\{y,y',y''\}$ differ in all three positions (i.e., do not form an edge of $\Gamma(S)$),
 while each element in $\{x,x',x''\}$ coincides with
 each element in $\{y,y',y''\}$ in one position,
 forming nine edges in total.
 This means that we have a $K_{3,3}$-subgraph
 of~$\Gamma(S)$.

 Next, assume that $\{x, x', x''\}$ and $\{y, y', y''\}$ are the bipartite parts of a $K_{3,3}$-subgraph of~$\Gamma(S)$.
 Let $x=(r_1,c_2,l_3)$,
 $x'=(r_2,c_3,l_1)$, and $x''=(r_3,c_1,l_2)$.
 Again, we see that $r_1$, $r_2$, $r_3$ are
 distinct, $c_1$, $c_2$, $c_3$ are
 distinct, and  $l_1$, $l_2$, $l_3$ are
 distinct too (otherwise, two of $x$, $x'$, $x''$ form an edge). Now, consider~$y$.
 It forms an edge with each of $x$, $x'$, $x''$, and hence coincides in one position
 with each of them.
 Without loss of generality (up to renaming $x$, $x'$, $x''$),
 we assume $y=(r_1,c_3,l_2)$.
 Similarly, $y'$ and $y''$ have
 $r_2$ and $r_3$ in the first position,
 $c_1$ and $c_2$ in the second position,
 $l_1$ and $l_3$ in the third position.
 Let $y'$ start with $r_2$ and $y''$ with $r_3$. Then $y'$ cannot have $l_1$ in the third position (otherwise, it coincides with $x'$ in two positions) and $y''$ cannot have $c_1$ in the second position (similarly, with $x''$). The only possibility
 is $y'=(r_2,c_1,l_3)$, $y''=(r_3,c_2,l_1)$, i.e., we have~\eqref{eq:213}.
\end{proof}

\begin{proposition}
 Given a finite group $(R,\cdot)$ of order at least~$3$, the value table $S=\{(r,c,r\cdot c)\mid r,c \in R\}$ is not $K_{3,3}$-free.
\end{proposition}
\begin{proof}
 First, we note that the group contains
 two distinct nonidentity commuting elements $x$ and~$y$.
 Indeed, take two different nonidentity elements $a$,~$b$.
 If $a\ne a^{-1}$, then $x=a$ and $y=a^{-1}$ satisfy the requirement;
 if $b\ne b^{-1}$, we take $x=b$, $y=b^{-1}$.
 If $a\cdot b=b\cdot a$, we take $x=a$, $y=b$.
 If $a= a^{-1}$, $b=b^{-1}$, and $a\cdot b \ne b\cdot a$, we take
 $x=a\cdot b$, $y=b\cdot a$.

 For such $x$ and $y$, the six triples
 $(\mathrm{id},x,x)$,
 $(\mathrm{id},y,y)$,
 $(x,\mathrm{id},x)$,
 $(x,y,x\cdot y)$,
 $(y,\mathrm{id},y)$,
 $(y,x,x\cdot y)$ induce a $K_{3,3}$-subgraph of $\Gamma(S)$.
\end{proof}

\section*{Acknowledgments}
Denis Krotov was supported by the grant of the Natural Science Foundation of Hebei Province (Project No.~A2023205045), by the 111 Center (Grant No.~D26018),
and within the framework of the state contract of the Sobolev Institute of Mathematics (FWNF-2026-0011).
%

\providecommand\href[2]{#2} \providecommand\url[1]{\href{#1}{#1}}
  \def\DOI#1{{\href{https://doi.org/#1}{https://doi.org/#1}}}\def\DOIURL#1#2{{\href{https://doi.org/#2}{https://doi.org/#1}}}

\end{document}